\documentclass[a4paper]{amsart}

\usepackage{esint}
\usepackage{MnSymbol}
\usepackage{enumitem}
\usepackage[english]{babel}

\usepackage[
	colorlinks,
	pdfpagelabels,
	pdfstartview = FitH,
	bookmarksopen = true,
	bookmarksnumbered = true,
	linkcolor = blue,
	plainpages = false,
	hypertexnames = false,
	citecolor = red] {hyperref}

\usepackage[capitalize]{cleveref}
\crefname{enumi}{}{}
\crefname{equation}{}{}

\makeatletter
\def\@tocline#1#2#3#4#5#6#7{\relax
  \ifnum #1>\c@tocdepth 
  \else
    \par \addpenalty\@secpenalty\addvspace{#2}%
    \begingroup \hyphenpenalty\@M
    \@ifempty{#4}{%
      \@tempdima\csname r@tocindent\number#1\endcsname\relax
    }{%
      \@tempdima#4\relax
    }%
    \parindent\z@ \leftskip#3\relax \advance\leftskip\@tempdima\relax
    \rightskip\@pnumwidth plus4em \parfillskip-\@pnumwidth
    #5\leavevmode\hskip-\@tempdima
      \ifcase #1
       \or\or \hskip 1em \or \hskip 2em \else \hskip 3em \fi%
      #6\nobreak\relax
    \dotfill\hbox to\@pnumwidth{\@tocpagenum{#7}}\par
    \nobreak
    \endgroup
  \fi}
\makeatother

\newtheorem{theorem}{Theorem}[section]
\newtheorem{proposition}[theorem]{Proposition}
\newtheorem{lemma}[theorem]{Lemma}
\newtheorem{corollary}[theorem]{Corollary}
\theoremstyle{definition}
\newtheorem{definition}[theorem]{Definition}
\newtheorem{remark}[theorem]{Remark}

\numberwithin{equation}{section}

\def \R {{\mathbb {R}}}

\def\grad{\nabla}

\def\supp{\operatorname{spt}}

\def\essliminf{\operatorname{essliminf}}
\def\dive{\operatorname{div}}
\newcommand{\dx}{\, {\rm d} x}
\newcommand{\dt}{\, {\rm d} t}
\newcommand{\de}{\, {\rm d}}

\newcommand{\abs}[1]{\left|#1\right|}
\newcommand{\norm}[1]{\left \lVert #1\right \rVert}
\newcommand{\Vnorm}[2][T]{\|#2\|_{\V(\Omega_{#1})}}
\newcommand{\Vprimenorm}[2][T]{\|#2\|_{\V'(\Omega_{#1})}}
\newcommand{\Wnorm}[2][T]{\|#2\|_{\mathcal{W}(\Omega_{#1})}}
\newcommand{\energy}[2][T]{\|#2\|_{\text{en},\Omega_{#1}}}

\newcommand{\V}{\mathcal{V}}
\newcommand{\W}{\mathcal{W}}

\newcommand{\varcappre}{\operatorname{cap'}_{\text{var}}}
\newcommand{\varcap}{\operatorname{cap}_{\text{var}}}

\newcommand{\encap}{\operatorname{cap}_{\text{en}}}
\newcommand{\capacity}{\operatorname{cap}}

\renewcommand{\hat}{\widehat}
\newcommand{\trm}{\textrm}
\newcommand{\ol }{\overline}

\allowdisplaybreaks

\begin{document} 
	

\author[B. Avelin and O. Saari]{Benny Avelin and Olli Saari}

\title[Characterizations of zero sets]
{Characterizations of interior polar sets for the degenerate $p$-parabolic equation}

\address{
	Benny Avelin,
	Department of Mathematics and Systems Analysis, 
	Aalto University School of Science,
	FI-00076 Aalto, 
	Finland
}
\address{
	Benny Avelin,
	Department of Mathematics, 
	Uppsala University,
	S-751 06 Uppsala, 
	Sweden
} 
\email{\color{blue} benny.avelin@math.uu.se}

\address{Olli Saari,	
	Department of Mathematics and Systems Analysis, 
	Aalto University School of Science,
	FI-00076 Aalto, 
	Finland
} \email{\color{blue} olli.saari@aalto.fi} 

\date{\today}



\begin{abstract}
	This paper deals with different characterizations of sets of nonlinear parabolic capacity zero, with respect to the parabolic $p$-Laplace equation. Specifically we prove that certain interior polar sets can be characterized by sets of zero nonlinear parabolic capacity. Furthermore we prove that zero capacity sets are removable for bounded supersolutions and that sets of zero capacity have a relation to a certain parabolic Hausdorff measure.
\end{abstract}
	
\maketitle 


\section{Introduction} The notion of capacity is central in potential theory. It is used for example when studying boundary regularity, characterizations of polar sets and removability results. In the nonlinear potential theory of elliptic equations, a capacity defined using the relevant Sobolev space has proved to be the correct way to measure exceptional sets with respect to supersolutions. For an account see \cite{HKM,KM} and \cite{Ma}. With different methods and definitions, the corresponding results have also been established for the heat equation and some of its generalizations, see \cite{EG,GZ1,GZ2,La1,La2} and \cite{W}. 

In this work we will consider the degenerate $p$-parabolic equation 
\begin{equation}
	\label{theeq} u_t - \dive( \abs{\nabla u}^{p-2} \nabla u ) = 0 , \quad 2<p< \infty\,. 
\end{equation}
The potential theory with respect to \cref{theeq} is different when compared to the heat equation, and many central problems are still open, e.g. a Wiener criterion. In addition it has not been obvious what is the most feasible way to define the capacity in this case and whether all the desired results can even be obtained using one definition. For some capacities related to the equation in \Cref{theeq} see \cite{DPP,S1} and \cite{S}.

In this work we will be concerned with the \emph{nonlinear parabolic capacity} that was suggested in \cite{KKKP}. That is for a set $E \subset \Omega_\infty = \Omega \times (0,\infty)$ they defined the \emph{nonlinear parabolic capacity} of $E$ as 
\begin{equation}
	\label{capaKKKP} \capacity(E,\Omega_\infty) = \sup \big\{\mu(\Omega_\infty): \supp \mu \subset E, 0 \leq u_\mu \leq 1\big\}\,, 
\end{equation}
where $\mu$ is a nonnegative Radon measure, and $u_\mu$ is a weak solution to the measure data problem
\begin{equation*} 
	\begin{cases}
		\partial_tu-\dive(\abs{\grad u}^{p-2}\grad u)=\mu & \text{in $\Omega_\infty$}, \\
		u(x,t) = 0 & \text{for $(x,t) \in \partial_p \Omega_\infty$.} 
	\end{cases}
\end{equation*}
In \cite{KKKP} they showed that the capacity defined in \cref{capaKKKP} makes sense, and used it to show an estimate on the superlevel sets of superparabolic functions. Considering the recent results in \cite{KuLiPa}, the superlevel set estimate in \cite{KKKP} gives that  the infinity set $E$ of a superparabolic function in the \emph{regular} class $\mathfrak{B}$\footnote{Class $\mathfrak{B}$ is all functions satisfying $u \in L^{p-1+p/n-\epsilon}_{\trm{loc}}$ for each $\epsilon > 0$, see \cite{KuLiPa}.}, has zero capacity in terms of the capacity in \cref{capaKKKP}. In the linear case the results in \cite{P} connects thermal capacity to quantities of variational nature. A corresponding result between variational parabolic capacity and nonlinear parabolic capacity was established in \cite{AKP}.

In this paper, we contribute to the theory by showing that the nonlinear parabolic capacity (see \cref{capaKKKP}) can actually be used to characterize certain interior polar sets, see \Cref{s.polar}. The result proves that the definition is in this sense optimal. Furthermore, it illustrates the fact that classical poles play some role in parabolic nonlinear potential theory. Intuitively the more reasonable objects would be the ''one-sided'' infinities defined in \cite{KuLiPa}. For instance, a solution with Dirac mass as a source, even though it is unbounded it is finite everywhere and thus it has no pole. 

In \Cref{s.balayage} we study the relation between the nonlinear parabolic capacity and the balayage. It is shown that the property of having capacity zero is equivalent to having an identically zero balayage for a very general class of subsets. 

In \Cref{s.removability} we prove removability results in terms of sets of zero nonlinear parabolic capacity. 

Finally in \Cref{s.Hausdorff} we provide dimensional estimates for polar sets with respect to certain parabolic Hausdorff type measures. Our proofs rely on an effective use of the tools developed in \cite{AKP} and \cite{KKKP}.

\section*{Acknowledgements} The first author was supported by the Swedish Research Council, \#637-2014-6822, the second author was supported by the V\"ais\"al\"a Foundation.

\section{Preliminaries}

\subsection{Parabolic spaces} We start by introducing some notions and notations. Our ambient space is $\mathbb{R}^{n+1}$ with a generic point denoted as $z = (x,t)$. The coordinates are called space and time, respectively. The reference domain $\Omega \subset \mathbb{R}^{n}$ is assumed to be smooth, and we denote $\Omega_T = \Omega \times (0,T)$. For any open set $U \subset \mathbb{R}^{n}$ we denote $U_{t_1,t_2} = U \times (t_1,t_2)$. By the parabolic boundary we mean the following subset of the topological boundary
\begin{equation}
	\nonumber \label{} 
	\partial_p \Omega_T = \big (\overline{\Omega} \times \{0\}\big )\cup \big (\partial \Omega \times (0,T]\big )\,. 
\end{equation}
Let us also define $a \approx b$ to mean that there is a constant $c(n,p)>1$ such that $c^{-1} a \leq b \leq c a$, and $\lesssim$ has the obvious meaning. If we use a subscript $a \approx_{q} b$ then the constant $c$ also depends on $q$.

For an open set $U \subset \mathbb{R}^{n}$ we denote by $W^{1,p}(U)$ and $W_0^{1,p}(U)$ the closures of $C^{\infty}(U)$ and $C_0^{\infty}(U)$ respectively, with respect to the Sobolev norm $\norm{f}_{L^{p}(U)} + \norm{ \nabla u}_{L^{p}(U)}$. The parabolic Sobolev space $L^{p}((0,T) ; W^{1,p}(U))$ consists of functions $u$ on $\mathbb{R}^{n+1}$ such that 
\begin{itemize}
	\item $x \mapsto u(x,t)$ is in $W^{1,p}(U)$ for a.e. $t \in (0,T)$,
	\item $t \mapsto \norm{u(\cdot,t)}_{W^{1,p}(U)}$ is measurable, and
	\item the norm
	\begin{equation*} 
		\norm{u}_{L((0,T); W^{1,p}(U))} := \left( \int_{0}^{T} \norm {u}_{W^{1,p}(U)}^{p} \, \dt \right)^{1/p} 
	\end{equation*}
	is finite. The local version is defined in a standard way.
\end{itemize}

\subsection{The equation} We study solutions of the degenerate $p$-parabolic equation
\begin{equation*}
	u_t - \dive( \abs{\nabla u}^{p-2} \nabla u ) = 0, \quad p > 2\,. 
\end{equation*}
We next define the first notion of sub and supersolution to be used in this paper.
\begin{definition}
	Let $\Theta \subset \R^{n+1}$ be an open set. A function $u$ is a weak supersolution in $\Theta$ if $u \in L^p(t_1,t_2,W^{1,p}(U))$ whenever $U_{t_1,t_2} \Subset \Theta$, and $u$ satisfies the inequality
	\begin{equation*}
		\int_{\Omega_T }( \abs{\nabla u}^{p-2} \nabla u \cdot \varphi - u \partial_t \varphi )\dx \dt \geq 0
	\end{equation*}
	for all nonnegative $\varphi \in C_0^{\infty}(\Omega_T)$. 
	It is a subsolution if the reverse inequality holds and a solution if equality holds.
\end{definition}
For a set of the form $\Omega_T = \Omega \times (0,T)$ we define boundary value zero at the parabolic boundary as follows: $u \in L^{p}((0,T); W_{0}^{1,p}(\Omega))$ and
\begin{equation*}
	\lim_{h \to 0} \frac{1}{h} \int_{0}^{h} \int_{\Omega} \abs{u}^{2} \dx \dt = 0\,. 
\end{equation*}
Recall that the lower semicontinuous regularization of a function $u$ is defined as
\begin{equation*} 
	\hat{u}(x,t) := \liminf_{(y,s) \to (x,t)} u(y,s)\,.
\end{equation*}
Here $\inf$ is the essential infimum.
\begin{theorem}[\cite{K}] \label{thm-regularization}
	Let $u$ be a weak supersolution in the domain $\Theta \subset \R^{n+1}$. Then the lower semicontinuous regularization $\hat u$ of $u$ is a weak supersolution and $u=\hat u$ almost everywhere in $\Theta$.
\end{theorem}

Another interesting class of supersolutions which contains\footnote{The lower semicontinous representative.} the weak supersolutions are the superparabolic functions, which we will define next. 
\begin{definition}
	\label{def_superparab} Let $\Theta \subset \R^{n+1}$ be a domain. A function $u : \Theta \to (-\infty, \infty]$ is superparabolic if 
	\begin{enumerate}
		[label=(\roman*)] 
		\item $u$ is lower semicontinuous,
		\item $u$ is finite in a dense subset of ${\Theta}$, and
		\item the following parabolic comparison principle holds: Let $U_{t_1,t_2}\Subset\Theta$ be a space-time cylinder, and let $h$ be a $p$-parabolic function in $U_{t_1,t_2}$ which is continuous in $\overline{U_{t_1,t_2}}$. Then, if $h\leq u$ on $\partial_p U_{t_1,t_2}$, $h\leq u$ also in $U_{t_1,t_2}$. 
	\end{enumerate}
\end{definition}

A few comments are in order. The sets $U$ can be assumed to be balls or arbitrary open sets; the same class of functions comes out anyway (see \cite{KKP}). Every supersolution is superparabolic after a redefinition in a set of measure zero (see \cite{K}), but there are superparabolic functions that do not belong to the correct function spaces in order to be supersolutions (see for example the Barenblatt solution or the friendly giant). We do however have the following result.

\begin{theorem}[\cite{KiLi,KKP}]
	Let $u$ be a locally bounded superparabolic function in the domain $\Theta \subset \R^{n+1}$, then $u$ is a weak supersolution in $\Theta$.
\end{theorem}

Note that the class of superparabolic functions is closed with respect to increasing sequences of superparabolic functions as long as the limit is finite a.e., a property the class of supersolutions does not have. 

Since we assumed $\Omega$ to be a smooth domain, we have from \cite{KiLuPa,KiLuPa1} that for every supersolution $u$ there is a Radon measure $\mu_u$ such that 
\begin{equation}
	\label{measprob} 
	\begin{cases}
		\partial_tu-\operatorname{div}(\abs{\grad u}^{p-2}\grad u)=\mu_u & \text{in $\Omega_\infty$}\,, \\
		u(x,t) = 0 & \text{for $(x,t) \in \partial_p \Omega_\infty$}
	\end{cases}
\end{equation}
in the sense of distributions. Conversely, for every Radon measure $\mu$ there is a supersolution $u_\mu$ solving \Cref{measprob} in the sense of distributions. There are, however, unbounded superparabolic functions that do not induce a Radon measure. For more on these aspects, see \cite{LiKi,KuLiPa}.

\subsection{Obstacle problem} \label{subsecobst} Solutions of obstacle problems will be important for us. 
\begin{definition}
	Let $\psi$ be a bounded measurable function, and consider the class
	\begin{equation*} 
		\begin{split}
			\mathcal S_{\psi}=\big \{\hat u\,:\, u &\textrm{ is a weak supersolution}, \hat u\geq\psi\ \text{in }\Omega_{T}\big \}\,. 
		\end{split}
	\end{equation*}
	Define the function
	\begin{equation*} 
		R^\psi(x,t)=\inf_{u \in \mathcal S_{\psi}} u(x,t)\,. 
	\end{equation*}
	We say that its lower semicontinuous regularization $\hat R^\psi$ is the solution to the obstacle problem with obstacle $\psi$. 
\end{definition}

If $\psi \in C(\overline \Omega_T)$, then the solution to the obstacle problem has the following properties. 
\begin{enumerate}
	[label=(\roman*)] 
	\item $\hat R^\psi \in C(\ol \Omega_T)$, 
	\item $\hat R^\psi$ is a weak solution in $\{(x,t)\in\Omega_T\,:\,\hat R^\psi(x,t)>\psi(x,t)\}$, and 
	\item $\hat R^\psi$ is the smallest superparabolic function above $\psi$ , i.e. if $v$ is superparabolic in $\Omega_T$ and $v\ge\psi$, then $v\ge \hat R^\psi$. 
\end{enumerate}
\begin{definition}
	Let $\psi$ be a bounded measurable function and let $U \subset \Omega_\infty$ be a set, then we define 
	\begin{equation}
		\nonumber \label{} R^\psi_U := R^{(\psi\, 1_U)}\,. 
	\end{equation}
	We call $R^{\psi}_U$ the reduced function of $\psi$ relative to $U$ in $\Omega_\infty$ or the r\'eduite. The lower semicontinuous regularization $\hat R^\psi_U$ is called the balayage. In the case when $\psi = 1$ we denote $R_U := R^1_U$. 
\end{definition}

For more about parabolic obstacle problems, see e.g. \cite{KKS} and \cite{LP}.

\subsection{Capacity and capacitability} 
In this section we will be concerned with the different notions of capacity that we will be using in this paper, and their relation to each other. We begin by noting that the \emph{nonlinear parabolic capacity} defined in \Cref{capaKKKP}, will henceforth be referred to simply as the \emph{parabolic capacity}. 

Next is a definition of inner and outer capacity which is classical in potential theory, see e.g. the classical book \cite{ChoqI}.
\begin{definition}
	Let $C(\cdot): 2^{\R^{n+1}} \to [0,\infty]$ be a set function, we define the inner capacity of an arbitrary set $E \subset \R^{n+1}$ as 
	\begin{equation}
		\nonumber \label{} \underline{C}(E) := \sup \big \{C(K): K \subset E, \text{$K$ compact}\big \}\,. 
	\end{equation}
	and the outer capacity as 
	\begin{equation}
		\nonumber \label{} \overline{C}(E) := \inf\big \{\underline{C}(U): U \supset E, \text{$U$ open}\big \}\,. 
	\end{equation}
\end{definition}
There are a few definitions of capacitability, which we will next briefly discuss. We start with Choquet's definition of capacitability, see for example \cite[Sec 6.4]{Choq}. 
\begin{definition}
	Let $C(\cdot): 2^{\R^{n+1}} \to [0,\infty]$ be a set function, we call a set $E \subset \R^{n+1}$ capacitable with respect to $C(\cdot)$ if the following holds 
	\begin{equation}
		\nonumber \label{} C(E) = \underline{C}(E)\,. 
	\end{equation}
\end{definition}
Note that in \cite{KKKP} it was proved that all Borel sets where capacitable with respect to the parabolic capacity. However what we need is that sets are capacitable with respect to the outer capacity of the parabolic capacity, that is, with respect to $\overline{\capacity}$. Thus we reserve the word ``capacitable'' for this case, which is the most important for us. 
\begin{definition}
	\label{defnonlincap} We call a set $E \Subset \Omega_\infty$ \emph{capacitable} if it is capacitable with respect to $\overline{\capacity}(\cdot,\Omega_\infty)$. 
\end{definition}
\begin{remark}
	\label{simplifyingcapacitability} 
	Let $K \subset \Omega_\infty$ be a compact set, then since open sets are capacitable with respect to the parabolic capacity and since the parabolic capacity is stable with respect to decreasing sequences of compact sets we can do the calculation
	\begin{align} \label{compact capacitability}
		\overline {\capacity}(K,\Omega_\infty) &=  \inf\{\underline {\capacity}(U,\Omega_\infty): U \supset K, \text{$U$ open}\} \notag \\
		&= \inf\{\capacity(U,\Omega_\infty): U \supset K, \text{$U$ open}\} = \capacity(K,\Omega_\infty)\,.
	\end{align}
	It now follows that 
	\begin{equation}
		\nonumber \label{} \underline {\Big[ \overline{\capacity} \Big ]}(E,\Omega_\infty) = \underline {\capacity}(E,\Omega_\infty) = \capacity(E,\Omega_\infty)\, 
	\end{equation}
	must hold for \textit{all} Borel sets (recall that the capacity is inner regular). Hence the capacitability condition of \Cref{defnonlincap} can be rewritten as 
	\begin{equation}
		\label{simpl} \inf\{\capacity(U,\Omega_\infty): U \supset E, \text{$U$ open}\} = \capacity(E,\Omega_\infty)\,.
	\end{equation}
	for a Borel set $E$.
	This is the relation that is needed for the rest of the paper, and thus when we reference \Cref{defnonlincap} we will be using \cref{simpl}. 
\end{remark}
\begin{remark}
	It is an open question whether every Borel set is capacitable in the sense of \Cref{defnonlincap}. From \Cref{simplifyingcapacitability} we see that according to the theory developed by Choquet \cite[Theorem 9.7]{ChoqI}, it would be enough to show strong subadditivity of the capacity. 
\end{remark}

The parabolic capacity is related with the two closely connected capacities: variational parabolic capacity and parabolic energy capacity. In order to define them, we need some notation. The relevant parabolic function space is
\begin{equation*}
	\mathcal{V}(\Omega_\infty) = L^{p}(0, \infty ; W_{0}^{1,p} (\Omega))\,,
\end{equation*}
we denote its dual by $\V'(\Omega_\infty)$. First define the norms 
\begin{align*}
	\norm{ u }_{\mathcal{V}(\Omega_\infty) } &= \left(\int_{\Omega_\infty} \abs{\nabla u }^{p} \dx \dt \right)^{1/p}, \\
	\norm{ v }_{\mathcal{V'}(\Omega_\infty) } &= \sup_{\phi \in C_0^\infty(\Omega \times \R)\,,\norm{ \phi }_{\mathcal{V}(\Omega_\infty) } \leq 1} \left | \int_{\Omega_\infty} v \phi \dx \dt \right |, 
\end{align*}
and then define the space $\W(\Omega_\infty) = \big \{u: u \in \V(\Omega_\infty) \text{ and } u_t \in \V'(\Omega_\infty)\big \}$. Using the quantity 
\begin{align*}
	\norm{ u }_{\mathcal{W}(\Omega_\infty) } &= \norm{ u }_{\mathcal{V}(\Omega_\infty) }^{p} + \norm{ u_t }_{\mathcal{V'}(\Omega_\infty) }^{p'}, 
\end{align*}
we can finally define the variational parabolic capacity
\begin{equation*}
	\varcappre(K,\Omega_\infty) = \inf \big\{ \norm{\varphi}_{\mathcal{W}} : \varphi \in C_0^{\infty}(\Omega \times \mathbb{R}), \varphi \geq 1_{K} \big \} 
\end{equation*}
for a compact set $K \subset \Omega_\infty$. The definition is extended to more general sets in the usual way. Namely, for each Borel set $B \Subset \Omega_\infty$ we let 
\begin{equation}
	\nonumber \label{} \varcap(B,\Omega_\infty) := \overline{\varcappre}(B,\Omega_\infty). 
\end{equation}
\begin{remark}
	\label{varcapextension} Note that this definition makes sense since the variational parabolic capacity is stable with respect to decreasing sequences of compact sets, see \cite[Lemma 3.1]{KKKP}. That is, for compact sets $K \subset \Omega_\infty$ 
	\begin{equation}
		\nonumber \label{} \varcappre(K,\Omega_\infty) = \varcap(K,\Omega_\infty) 
	\end{equation}
	and for open sets $U \Subset \Omega_\infty$ 
	\begin{equation}
		\nonumber \label{} \underline{\varcappre}(U,\Omega_\infty) = \varcap(U,\Omega_\infty)\,. 
	\end{equation}
\end{remark}

We next define the parabolic energy capacity but first we need to define the energy
\begin{equation*}
	\energy[\infty]{u} = \sup_{0<t<\infty} \int_{\Omega} u(x,t)^{2} \dx + \int_{\Omega_\infty} \abs{\nabla u}^{p} \dx \dt\,. 
\end{equation*}
The parabolic energy capacity of a compact set $K \subset \Omega_\infty$ is then defined as
\begin{equation*}
	\encap(K, \Omega_\infty) = \inf \big \{ \energy[\infty]{v} : v \geq 1_{K}, \ v \trm{ superparabolic}, v \in \mathcal{V}(\Omega_\infty) \big \}\,.
\end{equation*}

There is a relation between the three capacities, which we formulate in the following proposition. 
\begin{proposition}
	[\cite{AKP}] \label{lemma_zeroset} Let $K \subset \Omega_\infty$ be a compact set consisting of a finite union of space-time cylinders. Then
	\begin{equation*} 
		\capacity(K,\Omega_\infty)\approx \encap(K,\Omega_\infty)\approx \varcap(K,\Omega_\infty)\,. 
	\end{equation*}
	Moreover for a general compact set $K \subset \Omega_\infty$ we have
	\begin{equation*} 
		\capacity(K,\Omega_\infty)\approx \varcap(K,\Omega_\infty)\,. 
	\end{equation*}
	In particular, compact sets with vanishing capacity are the same for the parabolic capacity and the variational parabolic capacity. 
\end{proposition}

We are now ready to provide some examples of capacitable sets.
\begin{lemma}
	\label{lemma_capacitablesets} The following sets are capacitable: 
	\begin{enumerate}[label=(\arabic*)]
		\item \label{cap1} compact sets $K \subset \Omega_\infty$, 
		\item \label{cap2} open sets $U \Subset \Omega_\infty$, and
		\item \label{cap3} Borel sets $B \Subset \Omega_\infty$ for which $\varcap(B,\Omega_\infty)=0$. 
	\end{enumerate}
\end{lemma}
\begin{proof}
	First note that \Cref{cap1} was proved in \Cref{compact capacitability} and	\cref{cap2} is immediate from \Cref{simpl}.
	
	To prove \Cref{cap3}, it suffices to prove it for Borel sets $B$ which are neither open nor closed (compact). In this case by \Cref{varcapextension} we have 
	\begin{equation}
		\nonumber \label{} 0 = \inf \big \{ \varcap(U,\Omega_\infty): U \supset B, \text{ $U \Subset \Omega_\infty$ open} \big \}\,. 
	\end{equation}
	Now from \Cref{lemma_zeroset} and the definition of the involved quantities we get
	\begin{equation*} 
		\capacity(U,\Omega_\infty)\approx \varcap(U,\Omega_\infty) 
	\end{equation*}
	which implies 
	\begin{equation}
		\nonumber \label{} \capacity(B,\Omega_\infty) \leq \inf \big \{ \capacity(U,\Omega_\infty): U \supset B, \text{ $U \Subset \Omega_\infty$ open} \big \} = 0\,. 
	\end{equation}
	Hence $B$ is capacitable by \Cref{defnonlincap,simplifyingcapacitability}. 
\end{proof}
The following simple \namecref{corollar_generalcomparison} will be useful. 
\begin{corollary}
	\label{corollar_generalcomparison} If $B \Subset \Omega_\infty$ is a Borel set, then 
	\begin{equation}
		\nonumber \label{} \capacity(B,\Omega_\infty) \lesssim \varcap(B,\Omega_\infty)\,. 
	\end{equation}
	Furthermore if $B \Subset \Omega_\infty$ is a capacitable Borel set, then 
	\begin{equation}
		\nonumber \label{} \capacity(B,\Omega_\infty) \approx \varcap(B,\Omega_\infty)\,. 
	\end{equation}
\end{corollary}

\subsection{Useful results} Here we state some useful results that we are going to apply repeatedly. From \cite{AKP} we get the following useful \namecref{lemma_obsten}.
\begin{proposition}
	\label{lemma_obsten} Let $\psi \in C_0^{\infty} ( \Omega \times \mathbb{R} )$. If $u$ solves the obstacle problem in $\Omega_\infty$ with $\psi$ as the obstacle, that is, $u = \hat R^{\psi}$, then $u \geq \psi$ is a continuous supersolution and
	\begin{equation*} 
		\energy[\infty]{u} \lesssim \norm{\psi}_{\mathcal{W}(\Omega_\infty)}. 
	\end{equation*}
\end{proposition}

For the following two lemmas see \cite{KKKP} and the references therein. 
\begin{lemma}
	\label{lemma_supcon} Let $\{u_i\}_{i}$ be a sequence of uniformly bounded supersolutions such that $u_i \to u$ almost everywhere in $\Omega_\infty$. Then $u$ is a weak supersolution and $\mu_{u_i} \to \mu_u $ weakly. 
\end{lemma}
\begin{lemma}
	\label{lemma_compa} Let $K \subset \Omega_\infty$ be a compact set. Take $\lambda > 0$, then
	\begin{equation*}
		\mu_{R_K} (\Omega_\infty) \approx_{\lambda} \mu_{R^{\lambda}_K}(\Omega_\infty)\,. 
	\end{equation*}
\end{lemma}

\section{Polarity} \label{s.polar}

The peculiar nature of the potential theory of \cref{theeq} gives rise to superparabolic functions which do not have any finite measure as a source (see \cite{LiKi,KuLiPa}). In fact the superparabolic functions can be split into two very distinct classes, (see \cite{LiKi,KuLiPa}) 
\begin{align*}
	&\trm{Class $\mathfrak{B}$}  &u \in L_{loc}^{p-1+\frac{p}{n}-\epsilon} \quad  \trm{for each} \quad \epsilon > 0 \\
	&\trm{Class $\mathfrak{M}$}  &u \notin L_{loc}^{p-2+\epsilon} \quad  \trm{for each} \quad \epsilon > 0\,.
\end{align*}

With the above classes in mind we chose for technical reasons the following notion of polar sets.
\begin{definition} \label{defintpol}
	A Borel set $B \Subset \Omega_\infty$ is called an \emph{interior} polar set if there exists a superparabolic function $u$, such that for a large enough $\lambda_o > 0$ we have
	\begin{equation} \label{eqdefintpol}
		\big \{(x,t) \in \Omega_\infty: u(x,t) > \lambda_o \big \} \Subset \Omega_\infty
	\end{equation}
	and
	\begin{equation*}
		B \subset \big \{(x,t) \in \Omega_\infty: u(x,t) \equiv \infty \big \}\,.
	\end{equation*}
\end{definition}

\begin{remark}
	As was noted in \cite{KuLiPa} the assumption on $u$ in \Cref{defintpol} implies that $u \in \mathfrak{B}$.
\end{remark}

These interior polar sets will be treatable and we prove that:
\begin{enumerate}
	\item An interior polar set is a capacitable set of zero parabolic capacity, 
	\item A compactly contained, capacitable set of zero parabolic capacity is an interior polar set. 
\end{enumerate}
If we remove the assumption \Cref{eqdefintpol} we can include functions belonging to the class $\mathfrak{M}$, this would imply that the disc $B = B_{x_0} \times \{t_1\} \Subset \Omega_\infty$ is included (see \cite{KuLiPa}), but as was shown in \cite{AKP}, $B$ does not have zero parabolic capacity. The main reason for this is the fact that the functions in the class $\mathfrak{M}$ are ignored in the definition of parabolic capacity. 

\begin{remark}
	It should be noted that if $p=2$, then all compactly contained polar sets are interior polar sets.
\end{remark}

We will need the following theorem proved in \cite{KKKP}. 
\begin{theorem} \label{thm levelset bdd}
	Let $u \in \mathfrak{B}$ be superparabolic in $\Omega_\infty$ and $\lambda > 1$. Then there exists a constant $C > 1$, independent of $K$ and $\lambda$, such that 
	\begin{equation}
		\nonumber \label{} \capacity(\{u > \lambda\} \cap K,\Omega_\infty) \leq C \mu_{R^u_K}(\Omega_\infty)(\lambda^{1-p} + \lambda^{-1/(p-1)}) 
	\end{equation}
	for all compact sets $K \subset \Omega_\infty$. 
\end{theorem}

We immediately get that interior polar sets are sets of zero parabolic capacity.
\begin{theorem}
	\label{cor_polariszero} 
	Let $B \subset \Omega_\infty$ be an interior polar set, then it is a capacitable set of zero parabolic capacity.
\end{theorem}
\begin{proof}
	Since $B$ is an interior polar set, there is a superparabolic function $u$ such that \cref{eqdefintpol} holds. From \cref{eqdefintpol} we get that for a large enough $\lambda_0 > 0$ we have $K:=\overline{\{u > \lambda_0\}} \Subset \Omega_\infty$, thus from \Cref{thm levelset bdd} 
	\begin{equation} \nonumber \label{}
		\capacity(B,\Omega_\infty) \leq \capacity(\{u > \lambda\},\Omega_\infty) \leq C \mu_{R^u_K}(\Omega_\infty)(\lambda^{1-p} + \lambda^{-1/(p-1)})
	\end{equation}
	for $\lambda \geq \lambda_0$ and thus $\capacity(B,\Omega_\infty) = 0$, it is capacitable since $\{u > \lambda\}$ are open for each $\lambda$.
\end{proof}

The following result shows that there are plenty of superparabolic functions with honest classical poles, and that the interior polar sets are exactly the compactly contained, capacitable sets of zero parabolic capacity. The proof of the following theorem is based on a combination of a classical argument together with \Cref{lemma_obsten}.
\begin{theorem}
	\label{thm_zeroispolar} Let $B \Subset \Omega_\infty$ be a capacitable Borel set of zero parabolic capacity, then $B$ is an interior polar set. 
\end{theorem}
\begin{proof}
	Recall that the capacitability of $B$ implies that $\varcap(B,\Omega_\infty) = 0$, see \Cref{corollar_generalcomparison}. For each $j = 1,2, \ldots$ choose a function $\varphi_j \in C_0^{\infty}(\Omega \times \R)$ such that
	\begin{equation*}
		\Wnorm[\infty]{\varphi_j} \leq 2^{-j} 
	\end{equation*}
	and $\varphi_j \geq 1$ in a neighborhood of $B$. Let
	\begin{equation*}
		\hat \psi_m = \sum_{j=1}^{m} \varphi_j\,. 
	\end{equation*}
	For all $m \in \mathbb{N}$ we have $\Wnorm[\infty]{\hat \psi_m } \leq c_p $. Let $U \Subset \Omega_\infty$ be an open set such that $B \Subset U$, and consider a cutoff function $\eta \in C_0^\infty(U)$, then the function $\psi_m = \eta \hat \psi_m$, has $\Wnorm[\infty]{\psi_m } \leq c_p $ for a new constant $c_p > 1$, see e.g. \cite[Remark 2.3]{DPP}.
	
	Solving the obstacle problem with $\psi_m $ as the obstacle, we get a sequence of superparabolic functions $u_m $. By \Cref{lemma_obsten}
	\begin{equation*}
		\energy[\infty]{u_m } \leq C \Wnorm[\infty]{\psi_m } \leq C\, .
	\end{equation*}
	Since the sequence $\psi_m $ is increasing, it follows that the sequence $u_m$ is increasing. Note that the sequence $u_m$ has uniformly bounded energy and thus the limit function $u$ is finite a.e. and is hence a weak supersolution (see \cite[Remark 5.6]{KKP}). Note furthermore that $\supp \psi_m \subset U \implies \supp(\mu_{u_m}) \subset U$ and thus the measure corresponding to the supersolution $\mu$ has support inside $\overline{U} \subset \Omega_\infty$ by the weak convergence (see \cite[Remark 5.6]{KKP}).
	Note that because of the boundedness of the energies, the limit function is in $L^p(0,\infty;W^{1,p}_0(\Omega))$ and thus by classic regularity theory (see \cite{DiBe}) we have that for a large enough $\lambda$, $\{u > \lambda \} \Subset \Omega_\infty$ and 
	\begin{equation} \nonumber \label{}
		B \subset \big \{(x,t) \in \Omega_\infty: u(x,t) \equiv \infty\big \}\,.
	\end{equation}
	The above proves that $B$ is an interior polar set.
\end{proof}

\section{Balayage} \label{s.balayage} In this section we give a characterization of sets of zero capacity in terms of the obstacle problem (recall \Cref{subsecobst}). The proof is based on basic properties of the parabolic capacity and its relation to the variational parabolic capacity. 
\begin{theorem}
	\label{thm_balayage} Let $E \Subset \Omega_{\infty}$ be a bounded Borel set. Then the following two properties are equivalent: 
	\begin{enumerate}[label=(\arabic*)]
		\item \label{Bal1} The balayage $\hat R_E$ is identically zero. 
		\item \label{Bal2} The set $E$ is a capacitable set of zero parabolic capacity. 
	\end{enumerate}
\end{theorem}
\begin{proof}
	We begin by proving $\Cref{Bal1} \implies \Cref{Bal2}$. From \Cref{Bal1} we see that $R_E = 0$ almost everywhere. On the other hand, by Choquet's topological lemma (see e.g. \cite{HKM}), we may take a decreasing sequence of superparabolic functions $u_j \geq 1_E$, $j = 0,1,\ldots$ converging pointwise to $u$ with $\hat{u} \equiv \hat{ R}_E$. Since $\Omega$ is regular we can assume that $u_j$ vanishes continuously on $\partial_p \Omega_\infty$, and that $u_j \to 0$ as $t \to \infty$. Note that \Cref{lemma_supcon} implies that $u$ is a supersolution and that $\mu_{u_j} \to \mu_u$ weakly.
	
	Since each $u_j$ is lower semicontinuous we can define the open sets $U_j = \{u_j > 1/2 \} \supset E$ for each $j=0,1,\ldots$, and since $u_j$ vanishes continuously on $\partial_p \Omega_\infty$ we have $U_j \Subset \Omega_\infty$. Let us exhaust each $U_j$ with compact sets $K_j^{1} \subset K_j^{2} \subset \cdots \subset U_j$. By inner regularity of the capacity \cite[Lemma 5.5]{KKKP} we have 
	\begin{equation}
		\nonumber \label{} \lim_{i \to \infty} \capacity(K_j^{i}, \Omega_\infty) = \capacity(U_j, \Omega_\infty)\,. 
	\end{equation}
	Let $\varepsilon > 0$ be an arbitrary number. Take numbers $i_j(\epsilon)$ such that 
	\begin{equation}
		\nonumber \label{} \capacity(K_j^{i_j}, \Omega_\infty) + \varepsilon \geq \capacity(U_j, \Omega_\infty)\,, 
	\end{equation}
	and denote $K'_j = K_j^{i_j}$. Let $R_j = R^{1/2}_{K'_j}$, then using \Cref{lemma_compa} and $\mu_{R_{K'_j}}(K'_j) = \capacity(K'_j, \Omega_\infty)$ (see \cite[Theorem 5.7]{KKKP}), we have 
	\begin{equation}
		\label{Balest1} \capacity(U_j,\Omega_\infty) - \varepsilon \leq \capacity(K'_j,\Omega_\infty) = \mu_{R_{K'_j}}(\Omega_\infty) \lesssim \mu_{R_j}(\Omega_\infty)\,. 
	\end{equation}
	
	Note that $u_j \geq \hat{R}_j $ everywhere. Since $u_j \to 0$, we have $\hat{R}_j \to 0$, thus using \Cref{lemma_supcon} we get that $\mu_{\hat{R}_j} \to 0$ weakly. Since the supports of $\mu_{\hat R_j}$ are compactly contained in $\Omega_\infty$ we see by a standard characterization of weak convergence (e.g. \cite[Section 1.9]{EG1}) that $\mu_{\hat{R}_j}(\Omega_\infty) \to 0$. By monotonicity and \cref{Balest1},
	\begin{equation*}
		\capacity(E, \Omega_\infty) \leq \liminf_{j \to \infty} \capacity(U_j,\Omega_\infty) \leq \liminf_{j \to \infty} \mu_{\hat{R}_j}(\Omega_\infty) + \varepsilon = \varepsilon\,.
	\end{equation*}
	Since $\varepsilon > 0$ was arbitrary we have \Cref{Bal2}.
	
	Let us now prove that $\Cref{Bal2} \implies \Cref{Bal1}$. Take a decreasing sequence of open sets $U_j$, $j = 1, \ldots$ with capacities converging to zero and $U_j \supset E$. From \cite[Lemma 5.9]{KKKP} we get that 
	\begin{equation}
		\nonumber \label{} \mu_{R_{U_j}}(\Omega_\infty) = \capacity(U_j,\Omega_\infty)\,. 
	\end{equation}
	Since $\hat R_{U_j} \searrow u$ for some function $u$, we can use \Cref{lemma_supcon} to conclude that $\mu_u \equiv 0$ and thus $\hat u \equiv 0$, but note that by construction $\hat u \geq \hat R_E$ and thus $\hat R_E \equiv 0$. 
\end{proof}

\section{Removability}\label{s.removability} By \Cref{lemma_zeroset} the parabolic capacity is comparable to the variational parabolic capacity. Of all capacities this is the most convenient when dealing with removability results. However, compared to the stationary case, a further complication arises when dealing with the space $\W$ since it is not stable with respect to truncations (see the comment by Tartar in \cite{P}). As such, functions in a minimizing sequence for the capacity need not be less than $1$, and doing a ``soft'' truncation of a function in $\W$ introduces an $L^{1}$ error term in the time derivative. This forces us to work with a weaker norm in the removability results, and consequently we must assume more about our supersolutions, i.e. that they are bounded. It should be remarked that the zero capacity sets that we are concerned with here are larger than the zero sets of the capacity defined by Saraiva in \cite{S1}.

The stationary counterpart of the following lemma can be found in \cite{HKM}.
\begin{lemma}
	\label{lemma_testspace} Let $K \subset \Omega_\infty$ be a compact set. Then if $\capacity(K,\Omega_\infty) = 0$ we have
	\begin{equation} \nonumber \label{}
		\overline{C_0^\infty(O \setminus K)}_{\mathcal{S}(\Omega_\infty)} \eqsim \overline{C_0^\infty(O)}_{\mathcal{S}(\Omega_\infty)}\quad \trm{for all open $O \subset \Omega_\infty$}\,,
	\end{equation}
	where $\mathcal{S}$ is defined as
	\begin{equation} \nonumber \label{}
		\mathcal{S}(\Omega_\infty) = \lbrace u \in \V(\Omega_\infty), u_t \in \V'(\Omega_\infty) + L^1(\Omega_\infty) \rbrace\,.
	\end{equation}
\end{lemma}
\begin{proof}
	Assume first that $\capacity(K,\Omega_\infty) = 0$. Let $O \subset \Omega_\infty$ be an arbitrary open set, it is clear that 
	\begin{equation}
		\nonumber \label{} \overline{C_0^\infty(O \setminus K)}_{\mathcal{S}(\Omega_\infty)} \hookrightarrow \overline{C_0^\infty(O)}_{\mathcal{S}(\Omega_\infty)}\,. 
	\end{equation}
	To prove the other direction, take $\varphi \in C_0^{\infty}(O)$. Since $\capacity(K,\Omega_\infty) = 0$ and $K$ is capacitable, there is a sequence of open sets $U_j \supset K$ such that
	\begin{equation} \nonumber \label{}
		\lim_{j \to \infty} \capacity(U_j,\Omega_\infty) = 0\,.
	\end{equation}
	Thus, according to \Cref{corollar_generalcomparison} we have
	\begin{equation} \nonumber \label{}
		\lim_{j \to \infty} \varcap(U_j,\Omega_\infty) = 0\,,
	\end{equation}
	and thus there is a sequence of almost minimizers $u_j \in C_0^{\infty} (\Omega \times \R)$ satisfying
	\begin{align*}
		u_j \geq 1_{U_j} \quad \trm{and} \quad \Wnorm[\infty]{u_j} \leq 2\varcap(U_j,\Omega_\infty)\,.
	\end{align*}
	Inspired by \cite{Petitta} we let $\bar H(s) \in C_0^\infty \big((-1,1)\big )$ be a non-negative function such that $\int_0^1 \bar H(s) \de s = 1$. Let $H(s) = \int_0^s \bar H(s) \de s$, then $H(1) = 1$ and $H$ maps $\R_+ \cup\{0\} \to [0,1]$. Furthermore it is clear that if for a non-negative function $\phi \in C_0^\infty(\Xi)$, $\Xi \subset \R^{n+1}$, then $H(\phi) \in C_0^\infty(\Xi)$.
	
	Define $w_j = H(u_j) \in C_0^{\infty}(\Omega \times \R)$, then $w_j$ satisfies
	\begin{equation} \nonumber \label{}
		0 \leq w_j \leq 1 \quad \trm{and} \quad w_j \geq 1_{U_j}\,.
	\end{equation}
	
	Now as in \cite{Petitta}, we can split $\partial_t w_j = \bar H(u_j) \partial_t u_j$ into two parts, one of which will be in the dual space $\V'$ whereas the other one will be in $L^1$, that is,
\[ \partial_t w_j =  \left [ \partial_t w_j \right ]_a + \left [ \partial_t w_j \right ]_b  \]
with
\[ \left [ \partial_t w_j \right ]_a \in \mathcal{V}'(\Omega_\infty) \quad \trm{and} \quad \left [ \partial_t w_j \right ]_b \in L^{1}(\Omega_\infty).\]
	Indeed, since $\partial_t u_j \in \V'$, it can be represented as $\partial_t u_j = f_j - \dive F_j$ for some functions $f_j, F_j \in L^{p'}(\Omega_\infty)$, where $f_j$ is real valued and $F_j$ takes values in $\R^n$. Accordingly, we define
	\begin{equation} \nonumber \label{}
		\left [ \partial_t w_j \right ]_a = \left [ \bar H(u_j) \partial_t u_j \right ]_a := \bar H(u_j) (f_j - \dive F_j) - \bar H'(u_j) \grad u_j \cdot F_j\,.
	\end{equation}
	To calculate the dual norm of $\left [ \partial_t w_j \right ]_a$, let $v \in \V(\Omega_\infty)$. Then
	\begin{align*}
		\int_{\Omega_\infty} v \left [ \partial_t w_j \right ]_a \dx \dt &= \int_{\Omega_\infty} \left( v \bar H(u_j) (f_j - \dive F_j) - v \bar H'(u_j) \grad u_j \cdot F_j \right) \dx \dt \\
		&=: I_1 + I_2 + I_3\,.
	\end{align*}
	Let us first note that from H\"older's inequality we easily get
	\begin{equation} \label{eq:I1 norm}
		\abs{I_1} \leq \Vnorm[\infty]{v}\norm{f_j}_{L^{p'}(\Omega_\infty)}\,.
	\end{equation}
	The divergence theorem gives that 
	\begin{align*}
		I_2 + I_3 &= -\int_{\Omega_\infty} v \bar H(u_j) \dive F_j \dx \dt + I_3 \\
		&= \int_{\Omega_\infty} \left( \bar H(u_j) \nabla v \cdot F_j + v \bar H'(u_j) \grad u_j \cdot F_j \right) \dx \dt + I_3 \\
		&= \int_{\Omega_\infty} \bar H(u_j) \nabla v \cdot F_j \dx \dt\,.
	\end{align*}
	Thus by H\"older's inequality we get
	\begin{equation} \label{eq:I2 norm}
		\abs{I_2+I_3} \lesssim \Vnorm[\infty]{v} \norm{F_j}_{L^{p'}(\Omega_\infty)}\,.
	\end{equation}
	Combining \cref{eq:I1 norm,eq:I2 norm} we have proved that
	\begin{align*}
		\int_{\Omega_\infty} v \left [ \partial_t w_j \right ]_a \dx \dt \lesssim \Vnorm[\infty]{v}\Vprimenorm[\infty]{\partial_t u_j}\,,
	\end{align*}
	which implies that
	\begin{align} \nonumber \label{}
		\Vprimenorm[\infty]{\left [ \partial_t w_j \right ]_a} \lesssim \Vprimenorm[\infty]{\partial_t u_j}\,.
	\end{align}
	For the $L^{1}$ part, we can easily see that from H\"older's inequality
	\begin{equation} \nonumber \label{}
		\norm{\left [ \partial_t w_j \right ]_b}_{L^{1}(\Omega_\infty)} = \norm{\bar H'(u_j) \grad u_j \cdot F_j}_{L^{1}(\Omega_\infty)} \lesssim \Vnorm[\infty]{u_j} \Vprimenorm[\infty]{\partial_t u_j}\,.
	\end{equation}
	Finally, it is obvious that also $\Vnorm[\infty]{w_j} \lesssim \Vnorm[\infty]{u_j}$. Altogether, we have now proved that since the norms of $u_j$ go to zero, the following norm of $w_j$ also goes to zero:
	\begin{equation} \nonumber \label{}
		\Vnorm[\infty]{w_j} + \Vprimenorm[\infty]{\left [ \partial_t w_j \right ]_a} + \norm{\left [ \partial_t w_j \right ]_b}_{L^{1}(\Omega_\infty)} \to 0\,, \quad j \to \infty\,.
	\end{equation}
	
	Now $(1-w_j) \varphi \in C_0 ^{\infty}(O \setminus K)$ and
	\begin{equation*}
		(1-w_j) \varphi \to \varphi\,, \quad \trm{in $\mathcal{S}$.}
	\end{equation*}
\end{proof}

\begin{theorem} \label{thm removability}
	Let $K \subset \Omega_\infty$ be a compact set. Assume that $v$ is a weak solution (supersolution) in $\Omega_\infty \setminus K$, and that $v$ is bounded in a neighborhood of $K$. If $\capacity(K,\Omega_\infty) = 0$ then there is a continuous (lower semicontinuous) extension $u$ which is a weak solution (supersolution) in $\Omega_\infty$ and $v = u$ a.e in $\Omega_\infty \setminus K$. 
	
	On the other hand, if $\capacity(K,\Omega_\infty) > 0$ then there is at least one weak solution as above which cannot be extended to a weak solution in $\Omega_\infty$.
\end{theorem}

\begin{proof}
	We start by assuming $\capacity(K,\Omega_\infty) = 0$. Redefine and extend $v$ as follows
	\begin{equation} \nonumber \label{}
		u(x,t) := \underset{\Omega_\infty \setminus K \ni(y,s) \to (x,t)}{\essliminf} v(y,s)\,.
	\end{equation}
	Then the original $v$ coincides with the extension $u$ a.e. in $\Omega_\infty \setminus K$ (see \Cref{thm-regularization}).
	
	We will first prove that $u$ is in the correct parabolic Sobolev space. Let $O$, $K \subset O \subset \Omega_\infty$ be an open set such that $u \leq M$ on $O$. Let us cover $K$ with a finite union of space-time cylinders such that all cylinders are contained in $O$. Let $Q'$ be one of these cylinders, $Q' = B' \times (\tau'_1,\tau'_2)$, and consider an enlargement $Q \supset Q'$ such that $Q \subset O$ and denote $Q = B \times (\tau_1,\tau_2)$. Now consider a test function $\varphi \in C_0^\infty(Q)$ with $0 \leq \varphi \leq 1$, $\varphi = 1$ on $Q'$ and let $\varphi_j \in C_0^\infty(Q \setminus K)$, $0 \leq \varphi_j \leq 1$, be such that $\varphi_j \to \varphi$ in $\mathcal{S}(\Omega_\infty)$. Such a sequence exists because of \Cref{lemma_testspace}. Writing down the Caccioppoli estimate for weak supersolutions, \cite[Lemma 2.2]{KHar}, we get for $v = u+1$
	\begin{align} \label{eq_remcacc}
		\int_Q &|\nabla v|^p v^{-3/2} \varphi_j^p \dx \dt + \sup_{\tau_1 < t < \tau_2} \int_B v^{1/2} \varphi_j^p \dx\\
		&\lesssim \int_Q v^{p+3/2} |\nabla \varphi_j|^p \dx \dt +\abs{\int_Q v^{1/2} \frac{\partial \varphi_j^p}{\partial t} \dx \dt} =: I_1 + \abs{I_2}\,. \notag
	\end{align}
	Begin by noting that since $u \leq M$, we get
	\begin{equation} \label{eq_remgrad1}
		\int_Q |\nabla u|^p \varphi_j^p \dx \dt \lesssim_M \int_Q |\nabla v|^p v^{-3/2} \varphi_j^p \dx \dt\,.
	\end{equation}
	Since $u$ is bounded, we can deduce that
	\begin{align} \label{eq_remI2}
		\frac{\abs{I_2}}{p} &= \left \lvert \int_Q v^{1/2} \varphi_j^{p-1} \left [ \frac{\partial \varphi_j}{\partial t}\right ]_a \dx \dt+\int_Q v^{1/2} \varphi_j^{p-1} \left [ \frac{\partial \varphi_j}{\partial t} \right ]_b \dx \dt \right \rvert \notag \\
		&\lesssim \Vnorm[\infty]{v^{1/2} \varphi_j^{p-1}} \Vprimenorm[\infty]{ [\partial_t \varphi_j]_a}+\norm{ [\partial_t \varphi_j]_b}_{L^1(\Omega_\infty)} \notag \\
		&\lesssim \epsilon \Vnorm[\infty]{v^{1/2} \varphi_j^{p-1}}^p +  \frac{1}{\epsilon}\Vprimenorm[\infty]{ [\partial_t \varphi_j]_a}^{p'}+\norm{ [\partial_t \varphi_j]_b}_{L^1(\Omega_\infty)}\,,
	\end{align}
	where we in the last step have used Youngs inequality with $\epsilon$, and the subscripts $a$ and $b$ indicate the decomposition into the $\mathcal{V}'$ and $L^{1}$ parts. Moreover since $\varphi_j \leq 1$, $u \leq M$, and $v \geq 1$ we get for the first term on the right-hand side in \Cref{eq_remI2}
	\begin{align} \label{eq_remgrad2}
		\int_{Q} |\grad (v^{1/2} \varphi_j^{p-1})|^p \dx \dt &\lesssim \int_{Q} |\grad u|^p v^{-p/2} \varphi_j^{p(p-1)} \dx \dt + \int_{Q} |\grad \varphi_j|^p v^{p} \varphi_j^{p(p-2)} \dx \dt \notag\\
		&\lesssim_{M} \int_{Q} |\grad u|^p \varphi_j^{p} \dx \dt + \int_{Q} |\grad \varphi_j|^p \dx \dt\,.
	\end{align}
	Combining \Cref{eq_remcacc,eq_remgrad1,eq_remI2,eq_remgrad2} we see that taking $\epsilon$ small enough depending on $p,n,\delta,M$, we can reabsorb and get
	\begin{align} \label{eq_remlocalgradbdd}
		\int_Q |\nabla u|^p \varphi_j^p \dx \dt \lesssim_{M} \max\big \{\norm{\varphi_j}_{\mathcal{S}},\norm{\varphi_j}_{\mathcal{S}}^{p}\big \}+1
	\end{align}
	where the comparison constant depends on the $\Vnorm[\infty]{\phi_j}$. The inequality \Cref{eq_remlocalgradbdd}, the boundedness of $u$ in $O$ together with the fact that $u$ is a weak solution (supersolution) implies that $u \in L_{\trm{loc}}^{p}((0,\infty); W_{\trm{loc}}^{1,p}(\Omega))$\,.

	Next, let us show that the extended function $u$ is also a solution (supersolution) in $\Omega_\infty$. Again we consider the same covering of $K$ as above, and consider a $Q,Q'$ pair. Let $\psi \in C_0^{\infty}(Q)$. Then from \Cref{lemma_testspace} we see that there exists $\psi_j \in C_0^{\infty}(Q \setminus K)$ such that $\psi_j \to \psi$ in $\mathcal{S}(\Omega_\infty)$. Hence
	\begin{align*}
		&\left \lvert \int_{Q} \abs{\nabla u}^{p-2} \nabla u \cdot \nabla (\psi_i - \psi) \dx \dt \right \rvert \\
		& \lesssim_{|Q|} \left( \int_{Q} \abs{\nabla u}^{p} \dx \dt \right)^{1/p'} \left( \int_{\Omega_\infty} \abs{\nabla (\psi_i - \psi)}^{p} \dx \dt \right)^{1/p} \\
		&\lesssim_{|Q|} \norm{u}_{\V(Q)}^{p-1} \norm{\psi_i - \psi}_{\V(\Omega_\infty)} \to 0\,,
	\end{align*}
	and using that $u \leq M$ in $Q$
	\begin{align*}
		\abs{ \int_{Q} u \partial_t(\psi_i - \psi) \dx \dt } & 
		\leq \abs{ \int_{Q} u \left [ \partial_t(\psi_i - \psi) \right ]_a \dx \dt} + \abs{\int_{Q} u \left [\partial_t(\psi_i - \psi)\right ]_b \dx \dt} \\
		&\lesssim_M \norm{u}_{\V(Q)} \Vprimenorm[\infty]{\left [ \partial_t(\psi_i - \psi) \right ]_a } \\
		&\quad + \norm{\left [\partial_t(\psi_i - \psi) \right ]_b }_{L^1(\Omega_\infty)}\to 0\,,
	\end{align*}
	again the subscripts $a,b$ denotes the parts in $\V'$ and $L^1$ respectively.
	We thus see that for $\psi \in C_0^\infty(\Omega_\infty)$ we have
	\begin{align*}
		0 = (\leq) \int_{\Omega_\infty}( \abs{\nabla u}^{p-2} \nabla u \cdot \nabla \psi -& u \partial_t \psi )\dx \dt\,,
	\end{align*}
	which proves that $u$ is a solution (supersolution) in $\Omega_\infty$.
	
	Conversely, if $\capacity(K,\Omega_\infty) > 0$, then $\hat{R}_K $ is a bounded weak solution in $\Omega_\infty \setminus K$ with zero boundary values on $\partial_p \Omega_\infty$ that is not identically zero (see \Cref{thm_balayage}), and hence it does not have extension to a weak solution in $\Omega_\infty$.
\end{proof}

\section{Hausdorff estimates} \label{s.Hausdorff} In this section we relate the capacity to more geometric concepts. More precisely, we will prove Hausdorff estimates inspired by \cite{TW}. Before going into Hausdorff measures, we give an estimate which relates the standard Lebesgue measure of a set to the capacity of the set. 
\begin{lemma}
	Let $E$ be a compactly contained subset in $\Omega_\infty$, then the following inequality holds: 
	\begin{equation*}
		|E| \lesssim \capacity(E,\Omega_\infty)^{\frac{n+p}{n}}. 
	\end{equation*}
\end{lemma}
\begin{proof}
	First of all let $E$ be a finite collection of space-time cylinders. Let $\epsilon > 0$ be given and let $u$ be a superparabolic function admissible for the energy capacity such that it almost minimizes the energy, i.e. 
	\begin{equation*}
		\energy[\infty]{u} \leq \encap(E,\Omega_\infty) + \epsilon 
	\end{equation*}
	and $u \geq 1_E$. Note that $u$ can be assumed to be bounded. Using this and the parabolic Sobolev embedding, we see that 
	\begin{multline*}
		|E| \leq \int_{\Omega_\infty} u^{q} \dx \dt \lesssim \energy[\infty]{u}^{\frac{n+p}{n}} \lesssim \left (\encap(E,\Omega_\infty) + \epsilon\right )^{\frac{n+p}{n}} \lesssim \left ( \capacity(E,\Omega_\infty) + \epsilon \right )^{\frac{n+p}{n}}, 
	\end{multline*}
	where $q = p (n+p)/n $. We can then conclude the proof by outer regularity of the measure and the capacity. 
\end{proof}

In the above estimate we can only hope to get a one-sided bound, since for example discs have positive capacity but zero measure, however it should be noted that for parabolic cylinders of size $(r,r^p)$ this estimate is sharp. To improve the estimate above we start by introducing some notation and auxiliary results. Consider the translation invariant metric
\begin{equation*}
	d_p((x,t),(0,0) ) = \max \{\abs{x}, \abs{t}^{1/p} \}. 
\end{equation*}
In the definition of Hausdorff measure, we will replace the Euclidean distance by the distance defined above since if we want to get estimates through a Hausdorff measure, we must use a metric that distinguishes the time coordinate from the spatial ones. Let $Q_r(x,t) = \{(y,s): d_p((x,t);y,s) < r\}$ denote the metric ball with respect to the metric $d_p$ defined above and define the parabolic diameter as 
\begin{equation}
	\nonumber \label{} d_p(E) = \sup\{d_p((y,s);(x,t)),(y,s) \in E, (x,t) \in E\}. 
\end{equation}
Let $E \Subset \Omega_\infty$ and define the $s$-dimensional Hausdorff $\delta$-content with respect to the parabolic metric $d_p$ 
\begin{equation*}
	\mathcal{P}_\delta^{s} (E) = \inf \left \{ \sum_{i=1}^\infty d_p(A_i)^s, E \subset \bigcup_{i=1}^\infty A_i, A_i \subset \Omega_\infty, d_p(A_i) < \delta \right \} \,.
\end{equation*}
Taking the limit as $\delta \to 0$
\begin{equation*}
	\mathcal{P}^{s} (E) = \lim_{\delta \to 0} \mathcal{P}_\delta^{s} (E)
\end{equation*}
we obtain the Hausdorff measure with respect to the metric $d_p$.

To control the parabolic capacity with the Hausdorff measure above, we need an estimate for the capacity of parabolic cylinders. This is provided by \cite[Corollary 2]{AKP}, that is, if $Q_{2r}(x,t) \subset \Omega_\infty$, then 
\begin{equation}
	\label{eq_estcyl} \capacity( Q_r, \Omega_\infty ) \approx r^{n}. 
\end{equation}
Once we know this, we immediately get the following upper bound on the capacity in terms of the Hausdorff measure $\mathcal{P}^n$.
\begin{proposition}
	\label{prop_c<h} Let $B \Subset \Omega_\infty$ be a Borel set. Then $\capacity (B, \Omega_\infty) \lesssim_{n,p} \mathcal{P}^{n}(B)$. 
\end{proposition}
\begin{proof}
	Let $\epsilon > 0$ and fix $\delta > 0$. Let $\{E_i\}_i$ be a cover of $B$ such that $d_p(E_i) < \delta$ and
	\begin{equation*}
		\sum _i d_p(E_i)^{n} \leq \mathcal{P}_{\delta}^{n}(B) + \epsilon\,. 
	\end{equation*}
	For each $E_i$ choose a parabolic cylinder $C_i \supset E_i$ with radius $d_p(E_i)$. Now using subadditivity of the parabolic capacity and the estimate for the capacity of a parabolic cylinder \cref{eq_estcyl}, we get 
	\begin{align*}
		\capacity (B,\, \Omega_\infty) &\leq \capacity \big (\bigcup_i E_i ,\, \Omega_\infty \big ) \leq \sum_i \capacity ( C_i ,\, \Omega_\infty) \\
		& \lesssim \sum_i d_p(E_i)^{n} \lesssim \left(\mathcal{P}_{\delta}^{n}(B) + \epsilon \right). 
	\end{align*}
	The claim follows by taking $\delta \to 0$ and then $\epsilon \to 0$. 
\end{proof}

To prove a converse, we need a pair of lemmas. The overall idea is to use a modification of Frostman's lemma together with a potential estimate. The generalizations of Frostman's lemma to different metric spaces are well known, but in order to make our presentation as transparent as possible, we will not use the most general statements, but we will instead point out a certain shortcut that reduces the proof to the most classical case.
\begin{lemma}
	[Frostman, see \cite{BP,Howroyd}, and \cite{Mattila}] \label{lemma_frostman} Let $p > 2$ be a \textit{rational number}. Let $A \subset \mathbb{R}^{n+1}$ be a Borel (or even analytic) set with $\mathcal{P}^{s} (A) > 0$. Then there exists a Radon measure $\mu$ such that $\mu(\mathbb{R}^{n+1}) > 0$, $\supp (\mu) \subset A$ and
	\begin{equation*}
		\mu(E) \leq d_p(E)^{s} 
	\end{equation*}
	for all sets $E \subset \mathbb{R}^{n+1}$. Moreover, there exists a compact subset $K \subset A$ with $\mathcal{P}^{s} (K) > 0$. 
\end{lemma}
\begin{proof}
	We will briefly describe how the proof in an appendix of \cite{BP} also gives this theorem (and much more). First note that the rationality of $p = k/l$ implies the existence of $2^{l}$-adic parabolic rectangles. Indeed, if $Q_r$ is a parabolic rectangle with sides $r$ and $r^{p}$. We can partition it into $2^{n\,l + k}$ parabolic subrectangles of sides $2^{-l} r$ and $(2^{- l } r)^{p} = 2^{-k} r^p$, we call this collection $\mathcal{G}_1$. Proceeding inductively, assume that $\mathcal{G}_i$ is formed, each of its elements having sides $2^{-il}r$ and $2^{-ik} r^{p}$, we can repeat the procedure to obtain new parabolic rectangles having sides $2^{-(i+1)l}r$ and $2^{-(i+1)k} r^{p}$, forming $\mathcal{G}_{i+1}$. This inductive construction gives a generalized ''dyadic grid''.
	
	Once we have the grid of metric rectangles compatible with $d_p$, we can run the classical proof of Frostman's lemma for compact sets verbatim (see for instance \cite[Theorem 8.8]{Mattila} or \cite{BP}). It thus remains to show that a Borel (or analytic) set of positive $\mathcal{P}^{s}$ measure has a compact subset of positive measure. Again, a verbatim application of the proof in the appendix of \cite{BP} gives the result. 
\end{proof}
\begin{remark}
	The rationality assumption is not necessary. The theorem actually holds true in complete separable metric spaces (see \cite{Howroyd} and \cite{Mattila}). Moreover, the proof scheme above can also be implemented using a more general construction of dyadic cubes of metric spaces, originally due to Christ, see \cite{Christ} or \cite{HK}. However, the simple case of rational $p$ will be sufficient in our application. 
\end{remark}

The following potential estimate has been established in the elliptic case in \cite{KM}, and the parabolic version is from \cite{AKM} see also \cite{LSS} for a related result.
\begin{lemma}
	\label{lemma_pot} Let $u$ be a weak solution of the measure data problem
	\begin{equation*}
		u_t - \Delta_p u = \mu 
	\end{equation*}
	with finite Radon measure and zero boundary and initial values in $\Omega_\infty$. If 
	\begin{equation*}
		\left [\int_{0}^{r} \left( \frac{\mu(Q^-_\rho(x,t))}{\rho^{n}} \right)^{\frac{p}{n(p-2) + p}} \frac{{\rm d} \rho}{\rho} \right ]^{\frac{n(p-2) + p}{p}}
	\end{equation*}
	is locally bounded for some $r$, where $Q_\rho^-$ is the lower half of $Q_\rho$, then $u$ is locally bounded.
\end{lemma}

The potential estimate connects the properties of the measure to the boundedness of its potential. Geometrically, only sufficiently large sets can support measures that are not too badly concentrated in order to make the integral of the above lemma converge, and Frostman's lemma makes this principle quantitative in terms of Hausdorff measure. The following proposition glues these ideas together.
\begin{proposition}
	\label{prop_h<c} Let $s > n$. Let $A \Subset \Omega_\infty$ be Borel. If $\mathcal{P}^{s}(A) > 0$ then $\capacity(A,\Omega_\infty) > 0$. 
\end{proposition}
\begin{proof}
	Take $q \in \mathbb{Q}$. If $q > p$, then $d_{q}(E) \geq d_{p}(E)$ for all $E$ with small diameter. Consequently $\mathcal{P}^{s,d_p}(A) > 0$ implies $\mathcal{P}^{s,d_q}(A) > 0$. Using this condition with rational $q$, \Cref{lemma_frostman} ensures that there exists a nontrivial measure $\mu$ supported in $A$ such that 
	\begin{equation}
		\label{mu decay} \mu(Q^{p}_r) \lesssim  r^{p - q} \mu(Q^{q}_r) \leq r^{s -(q-p)} \quad \trm{for all} \ Q^q_{2r} \subset \Omega_\infty\, .
	\end{equation}
	Moreover there exists a compact set $K \subset A$ such that $\mu(K) > 0$.
	
	Let us make the counter assumption $\capacity(A,\Omega_\infty) = 0$. Let $\nu$ be a finite Radon measure with support inside $K$, and consider its corresponding lower semicontinuous potential, $u_{\nu}$ with zero boundary datum on $\Omega_\infty$. If $u_{\nu}$ is locally bounded then \Cref{thm removability} implies that $u_{\nu}$ can be extended across $K$ to a weak solution in $\Omega_\infty$. The maximum principle implies that the extension of $u_{\nu} \equiv 0$. Moreover since the original $u_{\nu}$ is lower semicontinuous the extension is exactly $u_{\nu}$, implying $\nu \equiv 0$. On the other hand, if $u_{\nu}$ is unbounded then \Cref{lemma_pot} implies that for any $r > 0$ there is a point $(x,t) \in \Omega_\infty$ such that
	\begin{equation} \nonumber \label{}
		\left [\int_{0}^{r} \left( \frac{\nu(Q^-_\rho(x,t))}{\rho^{n}} \right)^{\frac{p}{n(p-2) + p}} \frac{{\rm d} \rho}{\rho} \right ]^{\frac{n(p-2) + p}{p}} = \infty\,.
	\end{equation}
	 The above implies that for such $r$ and $(x,t)$ we cannot have
	\begin{equation} \nonumber \label{}
		 \nu(Q_{\rho}^-(x,t)) \lesssim \rho^{\sigma}, \quad \rho \leq r\,
	\end{equation}
for any $\sigma > n$. However since the measure $\mu$ constructed above restricted to $K$ satisfies $\mu(K) > 0$ and the decay estimate \Cref{mu decay} we see that our assumption that $A$ has zero parabolic capacity is false.
\end{proof}
We are now ready to collect our results, namely the following characterization of sets of zero parabolic capacity. With our weak methods we will not capture the borderline case $\mathcal{P}^{n}(A) > 0 \implies \capacity(A,\Omega_\infty) > 0$, it is not clear under which assumptions on the geometry of $A$ this implication holds for $p>2$. For example, in the case $p=2$, \emph{almost flat} sets support this implication (see \cite{TW}).
\begin{theorem}
	Let $A \subset \Omega_\infty$ be Borel. 
	\begin{itemize}
		\item If $\mathcal{P}^{n}(A) = 0$, then $\capacity (A,\Omega_\infty) = 0$. 
		\item If $\capacity (A,\Omega_\infty) = 0$, then $\mathcal{P}^{s}(A) = 0$ for all $s > n$. 
	\end{itemize}
\end{theorem}
\begin{proof}
	This follows directly from \Cref{prop_c<h,prop_h<c}. 
\end{proof}

\end{document}